\documentclass[12pt]{amsart}

\usepackage{amsfonts}
\usepackage{amsmath}
\usepackage{amssymb}
\usepackage{amsthm} 
\usepackage{mathrsfs}
\usepackage{xcolor}
\usepackage{appendix}
\usepackage{extarrows}

\newcommand\utimes{\mathbin{\ooalign{$\cup$\cr%
   \hfil\raise0.42ex\hbox{$\scriptscriptstyle\times$}\hfil\cr}}}
\newcommand\bigutimes{\mathop{\ooalign{$\bigcup$\cr%
   \hfil\raise0.36ex\hbox{$\scriptscriptstyle\boldsymbol{\times}$}\hfil\cr}}}
\usepackage[T1]{fontenc}
\usepackage{mathabx,graphicx}

\newcommand{\ab}{\allowbreak}
\theoremstyle{definition}
\newtheorem{thm}{Theorem}[section] 

\newtheorem{lem}[thm]{Lemma}
\newtheorem{prop}[thm]{Proposition}
\theoremstyle{definition}
\newtheorem{defn}[thm]{Definition}
\newtheorem{rem}[thm]{Remark}
\newtheorem{nota}[thm]{Notation}

\newcommand{\cA}{\mathcal{A}}

\newcommand{\cP}{\mathcal{P}}

\let\phi=\varphi
\let\epsilon=\varepsilon

\usepackage[colorlinks=true,urlcolor=blue,linkcolor=blue,citecolor=black]{hyperref}

\numberwithin{equation}{section} 

\author[Tseng]{Pei-Lun Tseng}
\address[P. -L. Tseng]{Department of Mathematics\\ New York University Abu Dhabi\\
Saadiyat Marina District, Abu Dhabi, United Arab Emirates
}
\email{pt2270@nyu.edu}

\title[A Unified approach to Infinitesimal Freeness]{A Unified approach to Infinitesimal Freeness with Amalgamation}

\begin{document}

\begin{abstract}

We consider the infinitesimal freeness in the operator-valued framework, and we show that the operator-valued infinitesimal (OVI) free independence is equivalent to the operator-valued free independence over an algebra of $2\times 2$ upper triangular matrices. We introduce the notion of OVI cumulants and investigate its properties, and we then deduce that the OVI freeness is equivalent to the vanishing of our mixed cumulants. Moreover, we derive the formula for obtaining the free additive and multiplicative convolutions within the realm of OVI freeness.

\end{abstract}

\maketitle

\tableofcontents

\date{}

\section{Introduction}
The idea of free independence was introduced by D. Voiculescu roughly 40 years ago. Since then, numerous extensions and generalizations have been discovered. One generalization of freeness, freeness of type B was found by Biane, Goodman, and Nica \cite{BGN} in 2003. A key observation in their paper was that freeness of type $B$ can be expressed to a certain extent in terms of Voiculescu's freeness with amalgamation \cite{ast} over $\widetilde{\mathbb{C}}$, a non-selfadjoint algebra of $2\times 2$ upper triangular Toeplitz  matrices. One of the motivations of this paper is to answer a question in Remark 4.9 in \cite{fevrier2010infinitesimal} as whether one can characterize infinitesimal freeness in terms of freeness over a two dimensional algebra. We show that tensoring with $\widetilde{\mathbb{C}}$ reduces infinitesimal freeness to ordinary freeness. One main result (Proposition \ref{thm1}) means all the results of \cite{nica2006lectures} can be easily carried over to the case of infinitesimal freeness of all orders; this will explored in a later paper.

The use in \cite{BGN} of upper triangular matrices inspired the development of infinitesimal freeness in \cite{belinschi2012free} and \cite{fevrier2010infinitesimal}. Thanks to the insight of Shlyakhtenko that
infinitesimal freeness may be used to detect the presence of spikes in many deformed random matrix models, it has become of interest to identify precisely which models have an asymptotic infinitesimal distribution, and which are
asymptotically infinitesimally free. In his pioneering work
\cite{DS}, Shlyakhtenko demonstrated that independent Gaussian and finite rank deterministic matrices, as well as independent Haar unitary and finite rank deterministic matrices, are asymptotically infinitesimally free (see \cite[Theorems 3.4 and 3.5]{DS}). Furthermore, aside from Shlyakhtenko's contributions, other researchers have established that additional random matrix models are infinitesimally free (see \cite{M}, \cite{DF}, \cite{AU21}, and \cite{PKPL22}). 


Another extension of free probability, initiated by Voiculescu \cite{ast}, is freeness with amalgamation (or operator-valued free probability), which parallels the classical concept of conditional independence. An application to random matrix theory was presented by Shlyakhtenko in \cite{SD}, where he showed that even though independent complex Gaussian band matrices are not asymptotically free, they are asymptotically free with amalgamation over diagonal matrices. In addition, this theory was applied by Belinschi, Mai, and Speicher \cite{BMS}
to describe the limit behavior of polynomial functions of various random matrix models via a linearization trick. 

Curran and Speicher \cite{CS} first introduced the concept of OVI infinitesimal freeness, demonstrating that Haar quantum unitary random matrices are asymptotically infinitesimal free over a unital $C^*$-algebra. Our aim in this paper is to delve into the infinitesimal free probability theory within the operator-valued framework. The key observation is 
motivated by \cite{BGN} that we find the connection between (operator-valued) infinitesimal freeness and operator-valued freeness. To be precise, we define the upper triangular probability space $(\widetilde{\mathcal{A}},\widetilde{\mathcal{B}},\widetilde{E})$ 
from a given OVI probability space $(\mathcal{A},\mathcal{B},E,E')$ and show that infinitesimal freeness for unital subalgebras $(\mathcal{A}_i)_{i\in I}$ in $(\mathcal{A},\mathcal{B},E,E')$ is equivalent to freeness for  
$(\widetilde{\mathcal{A}}_i)_{i\in I}$ in $(\widetilde{\mathcal{A}},\widetilde{\mathcal{B}},\widetilde{E})$ (Proposition \ref{thm1}). Based on this result, we have a way to switch problems from (operator-valued) infinitesimal freeness to freeness with amalgamation. We applied this technique to deduce several results in the realm of OVI freeness. 

First, we establish the equivalence between the OVI freeness of subalgebras and the vanishing of mixed operator-valued free cumulants and mixed OVI free cumulants (Theorem \ref{thm3}). Applying Proposition \ref{thm1}, we demonstrate the connection between OVI and matrix-valued OVI freeness (Proposition \ref{thm:matrix_equivalent}).
Additionally, we also show how to find the infinitesimal free additive and multiplicative convolutions of infinitesimal free self-adjoint random variables in the context of OVI probability spaces (Theorem \ref{thm4} and Theorem \ref{OVIMfcon}).
 
This paper is organized as follows. In Section 2, we provide a review of some fundamental concepts related to infinitesimal freeness and operator-valued freeness. In Section 3, we start to discuss the OVI probability. We consider notions of OVI free independence, cumulants, and study their basic properties. In Section 4, we focus on the construction of infinitesimal convolutions in the operator-valued framework. By applying Proposition \ref{thm1},  we demonstrate how to construct the OVI free additive and multiplicative convolutions.

\section*{Acknowledgement}

The author would like to thank Serban T. Belinschi, James A. Mingo, and the referee for their valuable discussion and comments. 

\section{Preliminaries}
\subsection{Infinitesimal Freeness}
We say $(\mathcal{A},\varphi,\varphi')$ is an \emph{infinitesimal probability space} if $\mathcal{A}$ is a unital algebra and $\varphi $ and $\varphi'$ are linear functionals from $\mathcal{A}$ to $\mathbb{C}$ such that $\varphi(1)=1$ and $\varphi'(1)=0$. 
Given an infinitesimal non-commutative probability space $(\mathcal{A},\varphi,\varphi')$ and a random variable $a\in \mathcal{A}$, a pair of linear functionals $(\mu,\mu')$ which both map $\mathbb{C}\langle X \rangle $ into $\mathbb{C}$ are called the \emph{infinitesimal distribution of $a$} if they satisfy $\varphi(P(a))=\mu(P)\ \mbox{and}\ \varphi'(P(a))=\mu'(P).$

In \cite{DS}, Shlyakhtenko showed how to create an infinitesimal probability space from an ensemble of random matrices. More precisely, suppose for each $N\in\mathbb{N}, A_1(N),A_2(N),\dots 
 ,A_k(N)$ are random matrices of size $N\times N$, and consider the map $\tau_N:\mathbb{C}\langle X_1,\cdots, X_k \rangle \rightarrow \mathbb{C}$ given by 
\begin{equation}
\tau_N(P)=E\Big[\frac{1}{N} Tr(P(A_1(N),A_2(N),\dots, A_k(N)))\Big], 
\end{equation}  where $\mathbb{C}\langle X_1,\dots, X_k \rangle$ is the algebra of polynomials in the $k$ non- \ab commuting variables $X_1,\dots,X_k$. Suppose $ \tau = \lim\limits_{N\rightarrow \infty} \tau_N $ exists; that is, $\lim_N\tau_N(P)=\tau(P)$ for all $P\in \mathbb{C}\langle X_1,\dots, X_k \rangle.$ 
If $\tau' = \lim\limits_{N\rightarrow \infty} N(\tau_N-\tau)$ exists, then we say that $\{A_1(N),A_2(N),\dots, A_k(N)\}_N$ has the limit infinitesimal distribution $(\tau,\tau').$
\begin{defn}
Let $(\mathcal{A},\varphi,\varphi')$ be an infinitesimal non-commutative probability space. We say the unital subalgebras $(\mathcal{A}_i)_{i\in I}$ of $\mathcal{A}$ are \emph{infinitesimally free} with respect to $(\varphi,\varphi')$ if for all $n\in \mathbb{N}$, $a_1,a_2,\cdots, a_n\in \mathcal{A}$ such that $a_k\in\mathcal{A}_{i_k}$ where $i_1, i_2, \cdots, i_n\in I$, $i_1\neq \cdots \neq i_n$, and $\varphi(a_1)=\varphi(a_2)=\cdots =\varphi(a_n)=0$, then we have 
\begin{eqnarray} \label{eqn1}
\varphi(a_1a_2\cdots a_n)&=& 0 ; \nonumber \\
\varphi'(a_1a_2\cdots a_n)&=& \sum\limits_{k=1}^n \varphi(a_1a_2\cdots a_{k-1}\varphi'(a_k)a_{k+1}\cdots a_n).
\end{eqnarray}  
\end{defn}
We note that the condition (\ref{eqn1}) on $\varphi'$ is equivalent to
\begin{eqnarray*}
\lefteqn{\varphi'(a_1\cdots a_n)} \\
&=& \left\{
        \begin{array}{lr}
        \varphi(a_1a_n)\varphi(a_2a_{n-1})\cdots \varphi(a_{(n-1)/2}a_{(n+3)/2})\varphi'(a_{(n+1)/2})  \\
        \ \mbox{if}\ n\ \mbox{is odd and}\ i_1=i_n,i_2=i_{n-1}\dots, i_{(n-1)/2}=i_{(n+1)/2} \\ 
        0 \ \mbox{otherwise} 
        \end{array}.
\right.
\end{eqnarray*}  
A set $\{a_i\mid i\in I\}$ is free if the unital subalgebras $alg(1,a_i)$ generated by $a_i(i\in I)$ form a infinitesimally free family. 

The notion of free and infinitesimal free cumulants, described in \cite{M} and \cite{fevrier2010infinitesimal}, plays a key role in characterizing infinitesimal freeness. Let us review this concept as follows.
\begin{nota} 
For $n\geq 1$, a partition of $[n]$ is a set $\pi=\{V_1,\dots,V_r\}$ of pairwise disjoint non-empty subsets of $[n]:=\{1,2,\dots,n\}$ such that 
$V_1\cup \cdots\cup V_r=[n]$. The set of all partitions of $[n]$ is denoted by $\cP(n)$, and $NC(n)$ denotes all non-crossing partitions of $[n]$ in the sense that we cannot find distinct blocks $V_r$ and $V_s$ with $a,c\in V_r$ and $b,d\in V_s$ such that $a<b<c<d$. 

Let $\cA$ be a unital algebra and $a \in \cA$. Suppose $a_1,\dots,a_n$ are elements in $\mathcal{A}$ and $V=\{i_1<\cdots<i_s\}$ is a block of some partition of $[n]$, we set $(a_1,\dots,a_n)|_V=(a_{i_1},\dots,a_{i_s}).$ Let $\{f_n:\mathcal{A}^n\to\mathbb{C}\}_{n\geq 1}$ and $\{f_n':\mathcal{A}^n\to\mathbb{C}\}_{n\geq 1}$ be sequences of multilinear functionals. For each partition $\pi$, we set
$$
f_{\pi}(a_1,\dots,a_n)=\prod_{V\in \pi}f_{|V|}((a_1,\dots,a_n)|_V).
$$
Moreover, if $V$ is a block in $\pi$, then $\partial f_{\pi,V}$ is defined by the map that is equal to $f_{\pi}$ except for the block $V$, where we replace $f_{|V|}$ by $f'_{|V|}$. Then we define $\partial f_{\pi}$ by
$$
\partial f_{\pi}(a_1,\dots,a_n)= \sum_{V\in \pi}\partial f_{\pi,V}(a_1,\dots,a_n).
$$
\end{nota}

\begin{defn}
Suppose $(\cA,\varphi,\varphi')$ is an infinitesimal probability space, the \emph{free cumulants} $\{\kappa_n:\mathcal{A}^n\to\mathbb{C}\}_n$ and \textit{infinitesimal free cumulants} $\{\kappa_n':\cA^n\to\mathbb{C}\}_n$ are defined inductively  via  
$$
\varphi(a_1\cdots a_n)=\sum_{\pi\in NC(n)}\kappa_{\pi}(a_1,\dots,a_n) \text{ and }
\varphi'(a_1,\dots,a_n)=\sum_{\pi\in NC(n)}\partial \kappa_{\pi}(a_1,\dots,a_n). 
$$
\end{defn}

Note that the free and infinitesimal free cumulants can also be described via 
$$
\kappa_n(a_1,\dots,a_n)=\sum_{\pi\in NC(n)}\mu(\pi,1_n)\varphi_{\pi}(a_1,\dots,a_n) \text{ and }
\kappa_n'(a_1,\dots,a_n)=\sum_{\pi\in NC(n)}\mu(\pi,1_n)\partial \varphi_{\pi}(a_1,\dots,a_n)
$$
where $\mu$ is the Möbius function of $NC(n)$ (see \cite[Lect. 11]{nica2006lectures} and \cite{fevrier2010infinitesimal,M}).
Note that the free and infinitesimal free cumulants can be used to characterize infinitesimal freeness as follows.
\begin{thm}[\cite{fevrier2010infinitesimal}]
Suppose that $(\cA,\varphi,\varphi')$ is an infinitesimal probability space, and $\cA_i$ is a unital subalgebra of $\cA$ for each $i\in I$. Then $(\cA_i)_{i\in I}$ are infinitesimally free if and only if for each $s\geq 2$ and $i_1,\dots,i_s\in I$ which are not all equal, and for $a_1\in \cA_{i_1},\dots,a_s\in \cA_{i_s}$, we have $\kappa_n(a_1,\dots,a_s)=\kappa_n'(a_1,\dots,a_s)=0.$
\end{thm}

\subsection{Operator-Valued Freeness}
Let $\mathcal{A}$ be a unital algebra and $\mathcal{B}$ be a unital subalgebra of $\mathcal{A}$. A linear map $E:\mathcal{A}\rightarrow \mathcal{B}$ is a \emph{conditional expectation} if 
\begin{eqnarray*}
&&E(b)=b\ \mbox{for all}\ b\in \mathcal{B}; \\
&&E(b_1ab_2)= b_1E(a)b_2 \ \mbox{for all}\ b_1,b_2\in\mathcal{B},\ a\in \mathcal{A}.
\end{eqnarray*} 
Then the triple $(\mathcal{A},\mathcal{B},E)$ is called an \emph{operator-valued probability space} (see \cite{ast}). 
\begin{defn}
Let $(\mathcal{A},\mathcal{B},E)$ be an operator-valued probability space, a family of subalgebras $(\mathcal{A}_i)_{i\in I}$ of $\mathcal{A}$ that contain $\mathcal{B}$ are \emph{free} with respect to $E$ over $\mathcal{B}$ if $E(a_1\cdots a_n)=0$ whenever $a_j\in \mathcal{A}_{i_j}$, $i_1\neq i_2\neq \cdots \neq i_n$, and $E(a_j)=0$ for all $j=1,\dots, n.$
\end{defn}

For an operator-valued probability space $(\mathcal{A},\mathcal{B},E)$, the \emph{operator-valued distribution} of a random variable $x\in \mathcal{A}$ is given by all 
operator-valued moments 
\[
E(xb_1xb_2\cdots b_{n-1}xb_nx)\in \mathcal{B}
\]
where $n\in\mathbb{N}$ and $b_1,\dots,\ b_n\in \mathcal{B}$. In other words, the operator-valued distribution of $x$ is the linear map $\nu:\mathcal{B}\langle X\rangle\to\mathcal{B}$ completely determined by
$$
\nu(Xb_1Xb_2\cdots b_{n}X) = E(xb_1xb_2\cdots b_{n}x)
$$
where $\mathcal{B}\langle X\rangle$ is the free algebra generated by an indeterminate variable $X$ over $\mathcal{B}$. 
Following \cite{S}, the \emph{free cumulants} $\kappa^{\mathcal{B}}_n:\mathcal{A}^n\rightarrow \mathcal{B}$ is defined by the moment-cumulant formula 
\begin{equation}
E(a_1\cdots a_n)=\sum\limits_{\pi\in NC(n)}\kappa^{\mathcal{B}}_{\pi}(a_1,\dots, a_n). 
\end{equation}
Note that the moment cumulants formula can also be defined by the following form
\begin{equation*}
\kappa^{\mathcal{B}}_n(a_1,\cdots,a_n)=\sum\limits_{\pi\in NC(n)}\mu(\pi,1_n)E_{\pi}(a_1,a_2,\cdots, a_n)
\end{equation*}
where $\mu$ is the M\"{o}bius function for $NC(n)$. Moreover, the notion of operator-valued freeness can be characterized by the vanishing of mixed cumulants property, which we stated it as follows. 

\begin{thm}[\cite{S}]
Suppose $(\mathcal{A},\mathcal{B},E)$ is an operator-valued probability space and $(\mathcal{A}_i)_{i\in I}$ are subalgebras of $\mathcal{A}$ contain $\mathcal{B}$. $(\mathcal{A}_i)_{i\in I}$ are free if and only if for each $s\geq 2,$ and $i_1,\dots,i_s\in I$ which are not all equal and for $a_1\in \mathcal{A}_{i_1},\dots,a_s\in \mathcal{A}_{i_s}$, we have $\kappa_s^{\mathcal{B}}(a_1,\dots,a_s)=0.$    
\end{thm}

For a given operator-valued probability space $(\cA,\mathcal{B},E)$, if we further assume that $\mathcal{A}$ is a unital $C^*$-algebra, $\mathcal{B}$ is a unital $C^*$-subalgebra, and $E$ is completely positive, then $(\mathcal{A},\mathcal{B},E)$ is called a \emph{$C^*$-operator-valued probability space}. 
For $x\in \mathcal{A}$,    
we say $x>0$ if $x$ is positive and invertible, and then the operator upper half plane is defined by $H^+(\mathcal{B})=\{b\in\mathcal{B} \mid \mathrm{Im}(b)=\frac{1}{2i}(b-b^*)>0\ \}.$ Note that if $x=x^*\in\mathcal{A}$ and $b\in H^+(\mathcal{B})$, then $b-x$ is invertible. 

For a fixed selfadjoint random variable $x\in \mathcal{A}$, the \emph{Cauchy transform of $x$} is defined by $$G_x(b)=E[(b-x)^{-1}]$$ 
for all $b\in H^+(\mathcal{B})$. Note that $G_x(b)$ is invertible for $b\in H^+(\mathcal{B})$. Then we let 
$$
F_x(b)=G_x(b)^{-1} \text{ and then } h_x(b)=F_x(b)-b.
$$ 

For each $n\in \mathbb{N}$, if we consider its fully matricial extension $G_x^{(n)}$ which is defined by
$$
G_{x}^{(n)}(b)=E\otimes 1_{n}[(b-x\otimes 1)^{-1}]
$$
for $b\in M_n(\mathcal{B})$ that $b-x\otimes 1$ is invertible, then it is known that the sequence $\{G_x^{(n)}\}_{n=1}^\infty$ encodes the operator-valued distribution of $x$. Note that $G_x^{(n)}$ is a holomorphic map that sends the upper half plane $H^+(M_n(\mathcal{B}))$ into the lower half plane $H^-(M_n(\mathcal{B})):=-H^+(M_n(\mathcal{B}))$. In analytic aspect, $G_x^{(n)}$ on $H^+(M_n(\mathcal{B}))$ essentially has the same behavior of $G_x=G_x^{(1)}$ on $H^+(\mathcal{B})$. We shall restrict our analysis on to $G_x$. 

In addition, the operator-valued $R$-transform and $S$-transform were first introduced in \cite{ast}. Then Dykema \cite{K} provided an new approach of these transforms. Let us sate the result as follows. Given $x=x^*\in\mathcal{A}$, there is an open subset $V_x$ of $H^-(\mathcal{B})\cap B(0,r)$ for some $r>0$ that $0\in \overline{V}_x$ such that
$G_x^{\langle -1 \rangle}$, the composition inverse of $G_x$, is well-defined on $V_x$. Thus, the \emph{$R$-transform of $x$} is defined by 
\begin{equation}\label{Reqn}
R_x(b)=G_x^{\langle -1\rangle}(b)-b^{-1} \text{ for all }b\in B(0,r).
\end{equation}
On the other hand, given element $x=x^*\in\mathcal{A}$, the moment generated function of $x$ is defined by
$$
\psi_x(b)=E\left((1-bx)^{-1}-1\right) \text{ for }\|b\|<\frac{1}{\|x\|}.
$$
Note that $\psi_x$ is Fr\'{e}chet analytic on a neighborhood of the origin and $\psi'_x(0)(\cdot)=(\cdot)E(x)$. Therefore if we further assume that $E(x)$ is invertible, then $\psi'_x(0)$ is invertible, which implies that $\psi_x$ is invertible around $0$ by the inverse function theorem. Then the \emph{$S$-transform of $x$} is defined by 
$$
S_x(b)=b^{-1} (1+b)\psi_x^{\langle-1\rangle}(b), \text{ for }\|b\| \text{ small enough.} 
$$
Suppose that $(\mathcal{A},\mathcal{B},E)$ is a $C^*$-operator-valued probability space, and $x=x^*$ and $y=y^*$ are two elements in $\mathcal{A}$. The free additive and multiplicative convolutions are given by the following statement. 
\begin{thm}
If $x$ and $y$ are free, then 
\begin{equation}\label{Rconvolution}
    R_{x+y}(b)=R_x(b)+R_y(b), \text{ for }\|b\| \text{ small enough} \qquad  \text{(see \cite{ast} and \cite{S}).}
\end{equation}
Moreover, if both $E(x)$ and $E(y)$ are invertible, then we also have
\begin{equation}\label{Sconvolution}
    S_{xy}(b)=S_y(b)S_x\left(S_y(b)^{-1}bS_y(b)\right), \text{ for }\|b\| \text{ small enough} \qquad  \text{(see \cite{K}).}\end{equation} 
\end{thm}
\begin{rem}
$(\mathcal{A},\mathcal{B},E)$ is called an \emph{operator-valued Banach probability space} if $\mathcal{A}$ is a unital Banach algebra, $\mathcal{B}$ is a subalgebra of $\mathcal{A}$ that containing $1_{\mathcal{A}}$, and $E:\mathcal{A}\to \mathcal{B}$ is a linear, bounded, $\mathcal{B}$-$\mathcal{B}$ bimodule projection. Note that \eqref{Rconvolution} and \eqref{Sconvolution} also hold if $(\mathcal{A},\mathcal{B},E)$ is only an operator-valued Banach non-commutative probability space. 
\end{rem}
In \cite{v2000}, Voiculescu provided us the subordination functions for operator-valued free additive convolution that we state as follows.     
\begin{thm}
For a $C^*$-operator-valued probability space $(\mathcal{A},\mathcal{B},E)$ and $x,y$ are selfadjoint random variables in $\mathcal{A}$ which are free, there exists a unique pair of Fr\'{e}chet analytic maps $\omega_1,\omega_2:H^+(\mathcal{B})\rightarrow H^+(\mathcal{B})$ such that \\
(1)\ $\operatorname{Im}(\omega_j(b))\geq \operatorname{Im}(b)$ for all $b\in H^+(\mathcal{B})$ and $j=1,2$\ ;\\
(2)\ $F_x(\omega_1(b))+b=F_y(\omega_2(b))+b=\omega_1(b)+\omega_2(b)$\ for all $b\in H^+(\mathcal{B})$\ ; \\
(3)\ $G_x(\omega_1(b))=G_y(\omega_2(b))=G_{x+y}(b)$\ for all\ $b\in H^+(\mathcal{B})$. 
\end{thm}

\section{Operator Valued Infinitesimal Probability}

The notion of OVI freeness was first introduced in \cite{CS}. In this section, we recall the definition of OVI probability spaces and the notion of infinitesimal freeness with amalgamation (Subsection \ref{21}). Then, we introduce the OVI cumulants and prove that the OVI freeness of subalgebras is equivalent to a vanishing condition for mixed cumulants and mixed infinitesimal cumulants in Subsection \ref{22}. Lastly, in Subsection \ref{23}, we will provide an application that extends the scalar version of infinitesimal freeness to the matrix version.

\subsection{OVI Freeness}\label{21}

In this subsection, we will begin by reviewing the definition of operator-valued infinitesimal probability spaces and the concept of infinitesimal freeness in the operator-valued setting. Additionally, we will introduce the notion of upper triangular probability spaces, which offers an alternative viewpoint on OVI freeness. This idea is inspired by \cite{BGN}. A crucial difference here is that we replace $\mathcal{A}$ by $\widetilde{\mathcal{A}}$; this is the part that was missing in \cite[Remark 4.8]{fevrier2010infinitesimal}.

Let $(\mathcal{A},\mathcal{B},E)$ be an operator-valued probability space (see \cite{ast}). Let $E':\mathcal{A}\rightarrow \mathcal{B}$ be a linear map such that $E'(1)=0$ and 
\begin{align}
E'(b_1ab_2)=b_1E'(a)b_2 \ \mbox{for all}\ b_1,b_2\in\mathcal{B}, a\in \mathcal{A}.
\end{align}
Then, $(\mathcal{A},\mathcal{B},E,E')$ is called an \emph{OVI probability space}. 

Given an OVI probability space $(\mathcal{A},\mathcal{B},E, \ab E')$ and $x\in\mathcal{A}$, the \emph{infinitesimal distribution} of $x$ is the pair of linear maps $(\nu,\nu')$ where $\nu$ is the distribution of $x$ and $\nu':\mathcal{B}\langle X\rangle\to\mathcal{B}$ is the map completely determined by
$$
\nu'(Xb_1Xb_2\cdots b_{k}X) = E'(xb_1xb_2\cdots b_{k}x).
$$

\begin{defn}
Given an OVI probability space $(\mathcal{A},\ab\mathcal{B},\ab E, E')$, the sub-algebras $(\mathcal{A}_i)_{i\in I}$ of $\mathcal{A}$ that contains $\mathcal{B}$ are called \emph{infinitesimally
free} with respect to $(E,E')$ (or \emph{OVI free}) if for $i_1,i_2,\dots,i_n\in I$, $i_1\neq i_2 \neq i_3\cdots\neq i_n$, and $a_j\in \mathcal{A}_{i_j}$ with $E(a_{j})=0$ for all $j=1,2,\dots,n$, the following two conditions hold: 
\begin{eqnarray*}
E(a_1\cdots a_n) &=& 0\ \ ; \\ 
E'(a_1\cdots a_n) &=& \sum\limits_{j=1}^n E(a_1\cdots a_{j-1}E'(a_j)a_{j+1}\cdots a_n). 
\end{eqnarray*}
\end{defn}
Since elements of $\mathcal B$ need not commute with the random variables $a_1,\dots, \ab a_n$, the factor $E'(a_j)$ may not be pulled out in front of $E$; however, we observe that second formula of infinitesimal freeness can be written as the follows:
\begin{equation}
E'(a_1\cdots a_n) =
           \left\{
              \begin{array}{lr}
                 E(a_1E(a_2(E(a_3\cdots E(a_{\frac{n-1}{2}}E'(a_{\frac{n+1}{2}})a_{\frac{n+3}{2}})\cdots a_{n-2})a_{n-1})a_n)  \\
                  \qquad\mbox{  if}\ n\ \mbox{is odd and}\ i_1=i_n, i_2=i_{n-1},\dots , i_{\frac{n-1}{2}}=i_{\frac{n+3}{2}} \\
              0\ \ \mbox{otherwise}\  \\
                  \end{array}
           \right. .
\end{equation}
This follows from known properties of freeness with amalgamation (see \cite{S,ast}). We said a set $\{a_i\mid i\in I\}$ is infinitesimally free with respect to $(E,E')$ if the unital algebras $alg(1,a_i)$ generated by $a_i (i\in I)$ form a infinitesimally free family.  

In fact, we have another point of view to see the OVI freeness which is related to upper triangular $2\times 2$ matrices. For a given OVI probability space $(\mathcal{A},\mathcal{B},E,E')$, we define sub-algebras $\widetilde{\mathcal{A}}$ and $\widetilde{\mathcal{B}}$ of $M_2(\mathcal{A})$ as follows
\[
\widetilde{\mathcal{A}}=
\left \{
\begin{bmatrix}
a & a' \\
0 & a 
\end{bmatrix} \Bigg |
a,a'\in\mathcal{A}
\right \}\ \mbox{and}\ \ 
\widetilde{\mathcal{B}}=
\left \{
\begin{bmatrix}
b & b' \\
0 & b 
\end{bmatrix} \Bigg |
b,b'\in\mathcal{B}
\right \}.
\]
Also, we define a map $\widetilde{E}$ from $\widetilde{\mathcal{A}}$ to $\widetilde{\mathcal{B}}$ by
\[
\widetilde{E}\Bigg( 
\begin{bmatrix}
a & a' \\
0 & a 
\end{bmatrix} 
\Bigg ) =
\begin{bmatrix}
E(a) & E(a')+E'(a) \\
0 & E(a) 
\end{bmatrix}.
\]
It is easy to see that $\widetilde{E}$ is a conditional expectation. This makes $(\widetilde{\mathcal{A}},\widetilde{\mathcal{B}},\widetilde{E})$ into an operator-valued probability space in the sense of \cite{ast}. We call it the \emph{upper triangular probability space} induced by $(\mathcal{A},\mathcal{B},E,E')$. Note that the algebras $\widetilde{\mathcal{A}}$ and $\widetilde{\mathcal{B}}$ are not selfadjoint, so we do not have a natural notion of positivity for $\widetilde{E}$.

\begin{rem}\label{rem1}
Let $(\mathcal{A},\mathcal{B},E,E')$ be an OVI probability space and $t$ be a variable such that $t^2=0$ and $ta=at$ for all $a\in \mathcal{A}$, then $\mathcal{A}+t\mathcal{A}$ is an two-dimensional unital algebra with the multiplication 
$$
(a+tb) \cdot (c+td) =(ac)+t(ad+bc),
$$ for $a,b,c$, and $d$ are in $\mathcal{A}$. Then,
$\mathcal{B}+t\mathcal{B}$ is a unital subalgebra of $\mathcal{A}+t\mathcal{A}$. Observe that 
$$
(E+tE')(a+tb)=E(a)+t(E(b)+E'(a))+t^2 E'(b) = E(a)+t (E(b)+E'(a)), 
$$ so that we shall define the map $(E+tE'):\mathcal{A}+t\mathcal{A}\rightarrow \mathcal{B}+t\mathcal{B}$ by 
$$
(E+tE')(a+tb)=E(a)+t(E(a')+E'(a)). 
$$ It is obvious that $E+tE'$ is a conditional expectation, and $(\mathcal{A}+t\mathcal{A},\mathcal{B}+t\mathcal{B},E+tE')$ is an operator-valued probability space. In fact, this is another way to describe the upper triangular probability space $(\widetilde{\mathcal{A}},\widetilde{\mathcal{B}},\widetilde{E})$. 
\end{rem}

Assume that $(\mathcal{A},\mathcal{B},E,E')$ is an OVI probability space and $(\widetilde{\mathcal{A}},\widetilde{\mathcal{B}},\widetilde{E})$ be its corresponding upper triangular probability space. The following proposition establishes the connection between these two spaces.   

\begin{prop} \label{thm1}
Sub-algebras $(\mathcal{A}_i)_{i\in I}$ that contain $\mathcal{B}$ are infinitesimally free with respect to $(E,E')$ if and only if $(\widetilde{\mathcal{A}}_i)_{i\in I}$ are free with respect to $\widetilde{E}$, where 
\[
\widetilde{\mathcal{A}}_i=
\left \{
\begin{bmatrix}
a & a' \\
0 & a 
\end{bmatrix} 
\Bigg |
a,a'\in\mathcal{A}_i
\right \}\ \mbox{for each}\ i\in I.
\]   
\end{prop}
\begin{proof}
First, we assume that $(\mathcal{A}_i)_{i\in I}$ are infinitesimally free. Suppose that $A_1,\dots, A_k$ are elements in $\widetilde{\mathcal{A}}$ such that $A_j\in \widetilde{\mathcal{A}}_{i_j}$ such that $\widetilde{E}(A_j)=0$ where $i_1,\dots, i_k\in I$ with $i_1\neq i_2 \neq \cdots \neq i_k.$ 
Note that for each $j$, 
\[ A_j=
\begin{bmatrix}
a_j & a'_j \\_
0 & a_j
\end{bmatrix} \ \mbox{for some}\ a_j,a'_j\in \mathcal{A}_{i_j}. 
\] 
Moreover, $\widetilde{E}(A_j)=0$ implies that $E(a_j)=0$ and $E(a'_j)=-E'(a_j).$
Note that
\begin{eqnarray*}
\lefteqn{\widetilde{E}(A_1\cdots A_k) =
\widetilde{E} \Bigg (
\begin{bmatrix}
a_1\cdots a_k & \sum\limits_{j=1}^k a_1\cdots a_{j-1}a'_ja_{j+1}\cdots a_k \\
0 & a_1\cdots a_k 
\end{bmatrix} \Bigg )} \\
&=&
\begin{bmatrix}
E(a_1\cdots a_k) & E'(a_1\cdots a_k)+\sum\limits_{j=1}^k E(a_1\cdots a_{j-1}a'_ja_{j+1}\cdots a_k )  \\
0                         & E(a_1\cdots a_k) 
\end{bmatrix}
\end{eqnarray*}
  
Observe that for each $j$, if we let $(a'_j)^\circ= a'_j-E(a'_j)$, then 
\begin{eqnarray*}
\lefteqn{E(a_1\cdots a_{j-1}a'_ja_{j+1}\cdots a_k )}\\
&=& E(a_1\cdots a_{j-1}(a'_j)^\circ a_{j+1} \cdots a_k)  
  +E(a_1\cdots a_{j-1}E(a'_j)a_{j+1}\cdots a_k) \\
&=& E(a_1\cdots a_{j-1}(a'_j)^\circ a_{j+1} \cdots a_k) 
  -E(a_1\cdots a_{j-1}E'(a_j)a_{j+1}\cdots a_k).
\end{eqnarray*}
By freeness, we have $E(a_1\cdots a_k)=0$ and also
$$E(a_1\cdots a_{j-1}(a'_j)^\circ a_{j+1}\cdots a_k)=0$$ for each $j$. 
Thus, $\widetilde{E}(A_1\cdots A_k)_{1,1}$ and $\widetilde{E}(A_1\cdots A_k)_{2,2}$ vanish and $\widetilde{E}(A_1\cdots A_k)_{1,2}$ can be rewritten as
\[
E'(a_1\cdots a_k)-\sum\limits_{j=1}^k E(a_1\cdots a_{j-1}E'(a_j)a_{j+1}\cdots a_k ), 
\]
which also vanishes by infinitesimal freeness. 

Conversely, suppose that $(\widetilde{\mathcal{A}}_i)_{i\in I}$ are free with respect to $\widetilde{E}$. Let $a_1,\dots,a_k$ be elements in $\mathcal{A}$ such that $a_j\in \mathcal{A}_{i_j}$ with $E(a_j)=0$ where $i_1,\dots, i_k\in I$ with $i_1\neq i_2 \neq \cdots \neq i_k.$ 

For each $j$, we define 
\[A_j=
\begin{bmatrix}
a_j & -E'(a_j) \\
0   & a_j
\end{bmatrix}.
\]
Then, it's obvious that $A_j\in\widetilde{\mathcal{A}}_{i_j}$ and  
\[
\widetilde{E}(A_j)=
\begin{bmatrix}
E(a_j) & E'(a_j)- E'(a_j) \\
0 & E(a_j)
\end{bmatrix} =
\begin{bmatrix}
0 & 0 \\
0 & 0
\end{bmatrix}.
\]
By freeness, we obtain $\widetilde{E}(A_1\cdots A_k)=0.$
Note that 
\begin{eqnarray*}
\lefteqn{0=\widetilde{E}(A_1\cdots A_k)}\\
&=&
\widetilde{E}
\left(  
\begin{bmatrix}
a_1\cdots a_k  & -\sum\limits_{j=1}^k a_1\cdots a_{j-1}E'(a_j)a_{j+1}\cdots a_k \\
0 & a_1\cdots a_k
\end{bmatrix}
\right)  \\
&=&
\begin{bmatrix}
E(a_1\cdots a_k)  & E'(a_1\cdots a_k)-\sum\limits_{j=1}^k E(a_1\cdots a_{j-1}E'(a_j)a_{j+1}\cdots a_k) \\
0 & E(a_1\cdots a_k)
\end{bmatrix}.
\end{eqnarray*} 
Hence, we obtain
\begin{eqnarray*}
E(a_1\cdots a_k)&=& 0\ ; \\
E'(a_1\cdots a_k) &=& \sum\limits_{j=1}^k E(a_1\cdots a_{j-1}E'(a_j)a_{j+1}\cdots a_k),
\end{eqnarray*}
which completes the proof. 
\end{proof}


\subsection{OVI Free Cumulants}\label{22}
For a given 
OVI probability space $(\mathcal{A},\mathcal{B},E,E')$, we have (operator- \ab valued) free cumulants $(\kappa^{\mathcal{B}}_n)_n$. In this section, we will go further to define the notion of (operator-valued) infinitesimal free cumulants and study its properties.  

Let $(\mathcal{A},\mathcal{B},E,E')$ be an OVI probability 
space. For a given $\pi \in NC(n)$ and $V\in \pi$, we consider the corresponding moment maps $\partial_V E_{\pi}$ which are defined just
as operator-valued moments $E_{\pi}$ associated to $\pi$ (see \cite[Sections 2.1 and 3.2]{S}), but replacing, for  block $V$,  $E$ by $E'$. Thus we define 
$$
\partial E_{\pi}(a_1,\dots,a_n)=\sum\limits_{V\in \pi}\partial_V E_{\pi}(a_1,\dots,a_n).
$$
For example, if $\pi =\{(1),(2,5),(3,4)\}$ and $V=\{(2,5)\}$, then $\partial_V E_{\pi}$ is given by
\[
\partial_V E_{\pi}(a_1,\dots,a_5)=E(a_1E'(a_2E(a_3a_4)a_5)).
\]
Also, if $V_1 = \{1\}$, $V_2 = \{2, 5\}$ and $V_3 = \{3, 4\}$ then
\begin{eqnarray*}
\partial E_{\pi}(a_1,\dots,a_5)
&=& \partial_{V_1} E_{\pi}(a_1,\dots,a_5)
+
\partial_{V_2} E_{\pi}(a_1,\dots,a_5)
+
\partial_{V_3} E_{\pi}(a_1,\dots,a_5) \\
&=& E'(a_1E(a_2E(a_3a_4)a_5))+E(a_1E'(a_2E(a_3a_4)a_5)) 
+
E(a_1E(a_2E'(a_3a_4)a_5)).
\end{eqnarray*}
\begin{defn} 
Suppose that $(\mathcal{A},\mathcal{B},E,E')$ is an OVI probability space. 
We define the \emph{OVI free cumulants} $\{\partial \kappa^{\mathcal{B}}_{n}:\mathcal{A}^n\rightarrow \mathcal{B}\}_{n\geq 1}$ to be the family of multilinear maps such that for all $n\in\mathbb{N}$ and $a_1,a_2,\ldots,a_n\in\mathcal{A}$,
\begin{equation}\label{eqn: OVI moment-cumulant formula}
\partial \kappa^{\mathcal{B}}_n(a_1,a_2,\ldots, a_n)=\sum\limits_{\pi \in NC(n)}\mu(\pi,1_n)\partial E_{\pi}(a_1,a_2,\ldots,a_n).  
\end{equation}
\end{defn}
Given an OVI probability space $(\mathcal{A},\mathcal{B},E, \ab E')$, we let $(\kappa^{\mathcal{B}}_n)_n$ and $(\partial\kappa^{\mathcal{B}}_n)_n$ be the free and infinitesimal free cumulants of $(\mathcal{A},\mathcal{B},E,\ab E')$.   Then consider the corresponding upper triangular probability space $(\widetilde{\mathcal{A}},\widetilde{\mathcal{B}},\widetilde{E})$ and their free cumulants $(\widetilde{\kappa}_n)_n$ of $(\widetilde{\mathcal{A}},\widetilde{\mathcal{B}},\widetilde{E}).$ 

We now state the following lemma that builds the connection between the OVI setting and the operator-valued setting in the cumulants aspect.
\begin{lem} \label{lem3}
Suppose that $A_1,\dots,A_n\in \widetilde{\mathcal{A}}$ with 
\[
A_i=\begin{bmatrix}
a_i & a'_i \\
0  & a_i
\end{bmatrix}
\] for each $i=1,\dots, n$. 
Then, 
\begin{multline}\label{kappa1}
\widetilde{\kappa}_n(A_1,\dots, A_n) = \\
\begin{bmatrix}
\kappa^{\mathcal{B}}_n(a_1,\dots, a_n) & \sum\limits_{j=1}^n\kappa^{\mathcal{B}}_n(a_1,\dots,a_{j-1},a'_j,a_{j+1},\dots, a_n)+\partial\kappa^{\mathcal{B}}_n(a_1,\dots,a_n) \\
0 & \kappa^{\mathcal{B}}_n(a_1,\dots, a_n)
\end{bmatrix}. 
\end{multline}
\end{lem}
\begin{proof}
It suffices to show that 
 for each $n\in \mathbb{N}$ and $\pi\in NC(n)$, we have 
\begin{eqnarray*}
\lefteqn{\widetilde{E}_{\pi}(A_1,\dots ,A_n) =} \\   
&&\kern-1em\begin{bmatrix}
E_{\pi}(a_1,\dots ,a_n) & \sum\limits_{j=1}^n E_{\pi}(a_1,\dots, a_{j-1},a'_{j},a_{j+1},\dots, a_n)+\partial E_{\pi}(a_1,\dots, a_n)\\
0 & E_{\pi}(a_1,\dots, a_n)
\end{bmatrix}.
\end{eqnarray*}
If we let $\pi\in NC(n)$ and $A_1,\cdots,A_n\in \widetilde{\mathcal{A}}$, by Remark \ref{rem1}, $\widetilde{E}_{\pi}(A_1,\dots,A_n)$ can be identified with $(E+tE')_{\pi}(a_1+ta'_1,\dots,a_n+ta'_n)$. Then it is clear that the constant term of $\widetilde{E}_{\pi}(A_1,\dots,A_n)$ is $E_{\pi}(a_1,\dots,a_n)$ and the term with first order is 
\[
E'_{\pi}(a_1,\dots,a_n)+\sum\limits_{j=1}^nE_{\pi}(a_1,\dots,a_{j-1},a'_j,a_{j+1},\dots,a_n), 
\] which complete the proof.
\end{proof} 
Now, we provide our main result in this section. 
\begin{thm}\label{thm3}
Given an OVI probability space $(\mathcal{A},\ab \mathcal{B}, \ab E,E')$, and 
$\mathcal{A}_1,\dots, \mathcal{A}_n$ are unital subalgebras of $\mathcal{A}$ that contain $\mathcal{B}$. Then the following two statements are equivalent: \\
(1). $\mathcal{A}_1,\dots, \mathcal{A}_n$ are infinitesimally free with respect to $(E,E').$ \\
(2). For every $n\geq 2$ and $i_1,\dots, i_s\in [n]$ which are not all equal, and for $a_1\in \mathcal{A}_{i_1},\dots, a_s\in\mathcal{A}_{i_s}$, we have $\kappa^{\mathcal{B}}_s(a_1,\dots,a_s)=\partial\kappa^{\mathcal{B}}_s(a_1,\dots,a_s)=0$.  
\end{thm}
\begin{proof} Assume that condition (1) is true. For $n\geq 2,$ we consider $a_j\in \mathcal{A}_{i_j}$ where $i_1,\dots,i_s\in [n]$ are not all equal. Then for each $j=1,\dots,s$, we let 
\[
A_j=
\begin{bmatrix}
a_j & 0 \\
0 & a_j
\end{bmatrix}. 
\]
It is obvious that $A_j\in \widetilde{\mathcal{A}}_{i_j}$ for each $j$. 
Now, by our assumption and Proposition \ref{thm1}, we obtain that $\widetilde{\mathcal{A}}_1,\dots,\widetilde{\mathcal{A}}_n$ are free with respect to $\widetilde{E}$, which implies that 
$\widetilde{\kappa}_s(A_1,\dots,A_s)=0$. 
Thus, by \eqref{kappa1}, we have
\[
0=\widetilde{\kappa}_s(A_1,\dots,A_s)=
\begin{bmatrix}
\kappa^{\mathcal{B}}_s(a_1,\dots,a_s) & \partial\kappa^{\mathcal{B}}_s(a_1,\dots,a_s) \\
0 & \kappa^{\mathcal{B}}_s(a_1,\dots, a_s)
\end{bmatrix}.
\]
Hence, we conclude that  $\kappa^{\mathcal{B}}_s(a_1,\dots,a_s)=\partial\kappa^{\mathcal{B}}_s(a_1,\dots,a_s)=0$. 

Conversely, we assume that condition (2) is true. We shall show that $\widetilde{\mathcal{A}}_1,\dots,\ab \widetilde{\mathcal{A}}_n$ are free with respect to $\widetilde{E}$ and then invoke Proposition \ref{thm1}.  
Fix $n\geq 2$, suppose that 
$A_1,\dots, A_n$ are elements in $\widetilde{A}$ such that $A_j\in \widetilde{\mathcal{A}}_{i_j}$ where $i_1,\dots, i_s\in [n]$ are not all equal. 
Note that each $A_j$ is of the form
\[
A_j=
\begin{bmatrix}
a_j & a'_j \\
0 & a_j
\end{bmatrix} \ \mbox{for some}\ a_j\ \mbox{and}\ a_j'\in \mathcal{A}_j.
\] 
We shall show that the $\mathcal{B}$-valued cumulants of $\mathcal{A}_1,\dots, \mathcal{A}_n$ are zero. 
By our assumption, we have 
\begin{eqnarray}
&&\kappa^{\mathcal{B}}_s(a_1,\dots,a_s)=0\ ; \\
&&\partial\kappa^{\mathcal{B}}_s(a_1,\dots,a_s)=0\ ; \\ &&\kappa^{\mathcal{B}}_s(a_1,\dots,a_{j-1},a'_{j},a_{j+1},\dots,a_s)=0\ \mbox{for each}\ j=1,\dots, s.
\end{eqnarray}
Thus by \eqref{kappa1}, we obtain $\widetilde{\kappa}_s(A_1,\dots, A_s)=0$. Hence, we deduce that
 $\widetilde{\mathcal{A}}_1,\dots,\widetilde{\mathcal{A}}_n$ are free with respect to $\widetilde{E}$, and then by Proposition \ref{thm1}, we have that $\mathcal{A}_1,\dots, \mathcal{A}_n$ are infinitesimally free with respect to $(E,E').$   
\end{proof}

\subsection{Matrix Valued Infinitesimal Freeness}\label{23}
We recall that if $(\mathcal{A},\varphi)$ is an
non-commutative probability space, and let $N\in \mathbb{N}$, then the triple 
$(M_N(\mathcal{A}),M_N(\mathbb{C}),E)$ with $ E=Id_N\otimes \varphi$ is an operator-valued probability space. There is a nice relation between scalar-valued and matrix-valued freeness, which we sate as follows: unital subalgebras $(\mathcal{A}_i)_{i\in I}$ are free with respect to $\varphi$ if and only if $(M_N(\mathcal{A}_i))_{i\in I}$ are free with respect to $E$ (see \cite[Chapter 10]{mingo2017free}). 

In fact, this result can be generalized to the operator-valued setting as in the following proposition. 
\begin{prop}\label{M-valued free prop}
Suppose $(\mathcal{A},\mathcal{B},E)$ be an operator-valued probability space and $N\in\mathbb{N}$. We consider the matrix-valued operator-valued probability space $(M_N(\mathcal{A}),M_N(\mathcal{B}),E^{(N)})$ where
$E^{(N)}:=Id_{N}\otimes E$. Then
unital subalgebras $(\mathcal{A}_i)_{i\in I}$ are free with respect to $E$ if and only if $(M_N(\mathcal{A}_i))_{i\in I}$ are free with respect to $E^{(N)}$.      
\end{prop}
The proof of Proposition \ref{M-valued free prop} can be done immediately by applying the following Lemma and utilizing the property of vanishing mixed free cumulants.

\begin{lem}\label{Mat-lem}
Suppose $\{a_{i,j}^{(1)}\},\dots,\{a_{i,j}^{(n)}\}$ are sets of elements in $\mathcal{A}$ and $\{b_{i,j}^{(1)}\},\dots,\{b_{i,j}^{(n-1)}\}$ are sets of elements in $\mathcal{B}$. If we set $A_j=[a_{i,j}^{(j)}]$ and $B_k=[b_{i,j}^{(k)}]$ for $j\in [n]$ and $k\in [n-1]$, then for any $r,s\in [N]$ we have  
\begin{eqnarray*}
&&\big[\kappa^{(E_N)}_n(A_1B_1,\dots,A_{n-1}B_{n-1},A_n)\big]_{r,s} \\
&=& 
\sum_{\substack{i_1,\dots,i_{n-1} \\ j_1,\dots,j_{n-1}\in [N]}}
\kappa_{n}^{(E)}(a_{r,i_1}^{(1)}b_{i_1,j_1}^{(1)},a_{j_1,i_2}^{(2)}b_{i_2,j_2}^{(2)},\dots,a_{j_{n-2},i_{n-1}}^{(n-1)}b_{i_{n-1},j_{n-1}}^{(n-1)},a_{i_{n-1},s}^{(n)}).
\end{eqnarray*}
\end{lem}
\begin{proof}
Given $r,s\in [N]$, note that
\begin{eqnarray*}
&&\big[\kappa^{(E_N)}_n(A_1B_1,\dots,A_{n-1}B_{n-1},A_n)\big]_{r,s} \\ 
&=& \big[\sum_{\pi\in NC(n)}\mu(\pi,1)E^{(N)}_{\pi}(A_1B_1,\dots,A_{n-1}B_{n-1},A_n)\big]_{r,s} \\
&=&  \sum_{\pi\in NC(n)}\mu(\pi,1) 
\sum_{\substack{i_1,\dots,i_{n-1} \\ j_1,\dots,j_{n-1}\in [N]}} 
E_{\pi}(a_{r,i_1}^{(1)}b_{i_1,j_1}^{(1)},a_{j_1,i_2}^{(2)}b_{i_2,j_2}^{(2)},\dots,a_{j_{n-2},i_{n-1}}^{(n-1)}b_{i_{n-1},j_{n-1}}^{(n-1)},a_{i_{n-1},s}^{(n)}) \\
&=& 
\sum_{\substack{i_1,\dots,i_{n-1} \\ j_1,\dots,j_{n-1}\in [N]}}
\sum_{\pi\in NC(n)}\mu(\pi,1) E_{\pi}(a_{r,i_1}^{(1)}b_{i_1,j_1}^{(1)},a_{j_1,i_2}^{(2)}b_{i_2,j_2}^{(2)},\dots,a_{j_{n-2},i_{n-1}}^{(n-1)}b_{i_{n-1},j_{n-1}}^{(n-1)},a_{i_{n-1},s}^{(n)})\\
&=& 
\sum_{\substack{i_1,\dots,i_{n-1} \\ j_1,\dots,j_{n-1}\in [N]}}
\kappa_{n}^{(E)}(a_{r,i_1}^{(1)}b_{i_1,j_1}^{(1)},a_{j_1,i_2}^{(2)}b_{i_2,j_2}^{(2)},\dots,a_{j_{n-2},i_{n-1}}^{(n-1)}b_{i_{n-1},j_{n-1}}^{(n-1)},a_{i_{n-1},s}^{(n)}).
\end{eqnarray*}
\end{proof}

We will show that there is an analogous proposition in the realm of OVI freeness. First, we note that for a given OVI probability space $(\mathcal{A},\mathcal{B},E,E')$ and $N\in\mathbb{N}$, it is easily to see that the quadruple $(M_N(\mathcal{A}),M_N(\mathcal{B}),E^{(N)},E^{(N)'})$
is also an OVI probability space where $E^{(N)}:=Id_N\otimes E$ and $E^{(N)'}:=Id_N\otimes E'$. In addition, we note that our upper triangular structure has the following nice isomorphism property.  
\begin{rem}\label{M-rem}
Note that for an algebra $\mathcal{A}$ and $N\in \mathbb{N}$, we have $M_N(\widetilde{\mathcal{A}})\cong \widetilde{M_N(\mathcal{A})}$ via the map 
\begin{equation*}
\begin{bmatrix}
a_{11} & a_{11}' & a_{12} & a_{12}' & \cdots & a_{1N} & a_{1N}' \\
0 & a_{11} & 0 & a_{12} & \cdots & 0 & a_{1N} \\
\vdots & &&&&& \vdots \\
a_{N1} & a_{N1}' & a_{N2} & a_{N2}' & \cdots & a_{NN} & a_{NN}' \\
0 & a_{N1} & 0 & a_{N2} & \cdots & 0 & a_{NN}
\end{bmatrix} \longmapsto \begin{bmatrix}
a_{11} & \cdots & a_{1N} & a_{11}' & \cdots & a_{1N}' \\
\vdots &  & \vdots & \vdots & & \vdots \\ 
a_{N1} & \cdots & a_{NN} & a_{N1}' & \cdots & a_{NN}' \\
0 & \cdots & 0 & a_{11} & \cdots & a_{1N} \\
\vdots &  & \vdots & \vdots & & \vdots \\ 
0 & \cdots & 0 & a_{N1} & \cdots & a_{NN}
\end{bmatrix}.
\end{equation*}
\end{rem}
\begin{prop}\label{thm:matrix_equivalent}
Consider an OVI probability space $(\mathcal{A},\mathcal{B},E,E')$ and $N\in\mathbb{N}$. Unital subalgebras $(\mathcal{A}_i)_{i\in I}$ are infinitesimally free with respect to $(E,E')$ if and only if $(M_N(\mathcal{A}_i))_{i\in I}$ are infinitesimally free with respect to $(E^{(N)},E^{(N)'})$. 
\end{prop}
\begin{proof}
Let $(\mathcal{A}_i)_{i\in I}$ be unital subalgebras of $\mathcal{A}$. Then 
\begin{eqnarray*}
&&(\mathcal{A}_i)_{i\in I} \text{ are infinitesimally with respect to }(E,E') \\ &\xLeftrightarrow{\textrm{Proposition }\ref{thm1}} &
(\widetilde{\mathcal{A}}_i)_{i\in I} \text{ are free with respect to }\widetilde{E} \\
&\xLeftrightarrow{\textrm{Proposition }\ref{M-valued free prop}} & 
(M_N(\widetilde{\mathcal{A}}_i))_{i\in I}\text{ are free with respect to }E^{(N)} \\
&\xLeftrightarrow{\textrm{Remark }\ref{M-rem}} & 
(\widetilde{M_N(\mathcal{A}_i)})_{i\in I}\text{ are free with respect to }E^{(N)} \\
&\xLeftrightarrow{\textrm{Proposition }\ref{thm1}} & 
(M_N(\mathcal{A}_i))_{i\in I}\text{ are infinitesimally free with respect to }(E^{(N)},E^{(N)'}). 
\end{eqnarray*}
\end{proof}
\begin{rem}
When $\mathcal{B}=\mathbb{C}$, Proposition \ref{thm:matrix_equivalent} shows that unital subalgebras $(\mathcal{A}_i)_{i\in I}$ are infinitesimally free in $(\mathcal{A},\varphi,\varphi')$ if and only if $(M_N(\mathcal{A}_i))_{i\in I}$ are infinitesimally free in $(M_N(\mathcal{A}),M_N(\mathbb{C}),Id_N\otimes \varphi,Id_N\otimes\varphi')$.
\end{rem}

\section{Operator-Valued Infinitesimal Free Convolutions } \label{3}

We will construct the infinitesimal free additive and multiplicative convolutions in this subsection. To do so, we will introduce the concepts of OVI Cauchy transform and OVI $S$-transform, and then demonstrate how to express the OVI additive (respectively multiplicative) convolutions in terms of OVI Cauchy transforms (respectively $S$-transforms).

\subsection{\texorpdfstring{$C^*$}{C*}-OVI Probability} 

In this subsection, we consider the analytic setting of OVI probability space $(\mathcal{A},\mathcal{B},E,E')$. Moreover,  
the notion of OVI Cauchy transform will be also introduced in this subsection.   

\begin{defn}\label{COVIspace}
$(\mathcal{A},\mathcal{B},E,E')$ is called a \emph{$C^*$-OVI probability space} if $(\mathcal{A},\mathcal{B},E)$ is a $C^*$-operator-valued probability space and $E':\mathcal{A}\rightarrow \mathcal{B}$ is a linear, $\mathcal{B}$-$\mathcal{B}$ bimodule, selfadjoint map that is bounded with $E'(1)=0$.   
\end{defn}

\begin{rem}\label{CIDRmk}
Following Definition \ref{COVIspace}, we also obtain scalar version of $C^*$-infinitesimal probability spaces if we set  $\mathcal{B}=\mathbb{C}$. More precisely,  $(\mathcal{A},\varphi,\varphi')$ is called a \emph{$C^*$-infinitesimal probability space} if $(\mathcal{A},\varphi)$ is a $C^*$-probability space and $\varphi':\mathcal{A}\rightarrow \mathbb{C}$ is a selfadjoint bounded linear functional with $\varphi'(1)=0$.  

Asymptotic infinitesimal distributions may be unbounded linear functionals; however, many asymp\-totic infinitesimal distributions are signed measures with compact support and bounded variation. These cases are covered by our Definition \ref{COVIspace}. For instance, the limit infinitesimal distributions of Gaussian Orthogonal Ensemble \cite{joh98} and complex Wishart matrices \cite{M} are such signed measures.
\end{rem}

Now, let us define the infinitesimal Cauchy transform of a given element. First, we state it as a formal series: 
for a given OVI probability space $(\mathcal{A},\mathcal{B},E,E')$ and given $x\in\mathcal{A}$, \emph{the infinitesimal Cauchy transform of $x$ at $b$} is given by 
$$
g_x(b):=E'((b-x)^{-1})=\sum\limits_{n\geq 0}E'(b^{-1}(xb^{-1})^n)
$$
whenever each formula makes sense. 

Note that if we consider matrix amplifications $(g_x^{(m)})_{m=1}^\infty$ of $g_x$, then $(g_x^{(m)})_{m=1}^\infty$ encodes all possible infinitesimal moments of $x$ where 
$$
g_x^{(m)}(b) := E'\otimes 1_m((b-x\otimes 1_m)^{-1})= \sum\limits_{n\geq 0}E'\otimes 1_m (b^{-1}(x\otimes 1_mb^{-1})^n)
$$
for each $m\in\mathbb{N}$ and $b\in M_m(\mathcal{B})$. The effectiveness of amalgamation over upper triangular matrices can be understood by observing that the resolvent is a non-commutative function (see \cite{kaliuzhnyi2014foundations}).

Now, let us assume that $(\mathcal{A},\mathcal{B},E,E')$ is a $C^*$-OVI probability space, and we define $g_x$ as follows. 
\begin{defn}
Suppose that $(\mathcal{A},\mathcal{B},E,E')$ is a $C^*$-OVI probability space 
and $x=x^*$ be an element in $\mathcal{A}$, the \emph{infinitesimal Cauchy transform} of $x$ is defined by
\[
g_x(b)=E'((b-x)^{-1}),\ \mbox{for all}\ b\in\mathcal{B}\ \mbox{with}\ b-x\ \mbox{is invertible.}
\] 
\end{defn}
It is clear that $g_x$ is well-defined on $H^+(\mathcal{B})$. Note that $g_x$ is analytic on $H^+(\mathcal{B})$ and for a given $b\in H^+(\mathcal{B})$, the Fr\'{e}chet derivative of $x$ at $b$ is given by
$$
g_x'(b)(\cdot)=-E'\left((b-x)^{-1}\cdot (b-x)^{-1}\right).
$$ 
Indeed, since $E'$ is bounded, there is $C>0$ so that 
$$
\|E'(a)\|\leq C\|a\| \text{ for }a\in\mathcal{A}.
$$
Thus, if we let 
$$
A_x(b)(\cdot)=-E'\left((b-x)^{-1}\cdot (b-x)^{-1}\right),
$$
then for $\|h\|$ small, we have
\begin{eqnarray*}
&&\|g_x(b+h)-g_x(b) -A_x(b)(h) \| \\
&=&\|E'\left((b+h-x)^{-1}\right)-E'\left( (b-x)^{-1}\right)-A_x(b)(h)\| \\
&=& \|E'\left( (b+h-x)^{-1}(b-x)(b-x)^{-1}\right) \\ &&-E'\left((b+h-x)^{-1}(b+h-x)(b-x)^{-1}\right)-A_x(b)(h) \| \\
&=& \| E'\left( (b+h-x)^{-1}(-h)(b-x)^{-1}\right)-A_x(b)(h)\| \\
&=& \|E'\left( [(b+h-x)^{-1}-(b-x)^{-1}]h(b-x)^{-1} \right) \|.
\end{eqnarray*}
Therefore,
\begin{eqnarray*}
&&\frac{ \|g_x(b+h)-g_x(b) -A_x(b)(h) \|}{\|h\|} \\
&=&\frac{\|E'\left( [(b+h-x)^{-1}-(b-x)^{-1}]h(b-x)^{-1} \right)\|}{\|h\|}  \\
&\leq& \frac{C\|[(b+h-x)^{-1}-(b-x)^{-1}]h(b-x)^{-1}\|}{\|h\|}  \\
&\leq& \frac{\|(b+h-x)^{-1}-(b-x)^{-1}\|\cdot\|h\|\cdot\|(b-x)^{-1}\|}{\|h\|}  \\
&=& \|(b+h-x)^{-1}-(b-x)^{-1}\|\cdot \|(b-x)^{-1}\| \longrightarrow 0 \text{ as }h\longrightarrow 0.
\end{eqnarray*}

If we further assume $\|b^{-1}\|<1/\|x\|$, then
$$
g_x(b)=E'(\sum\limits_{n\geq 0} b^{-1}(xb^{-1})^n)=\sum\limits_{n\geq 0} E'(b^{-1}(xb^{-1})^n),
$$
which coincides with $g_x$ as a formal series. 

\begin{rem}
Note that for a given $C^*$-OVI probability space $(\mathcal{A},\mathcal{B},E, \ab E')$ and $(\widetilde{\mathcal{A}},\widetilde{\mathcal{B}},\widetilde{E})$ be its upper triangular probability space. Then $(\widetilde{\mathcal{A}},\widetilde{\mathcal{B}},\widetilde{E})$ is an operator-valued Banach non-commutative probability space with the following norm on $\widetilde{\mathcal{A}}$: 
$$
\|A\|_{\widetilde{\mathcal{A}}}:=\|a\|+\|a'\| 
\text{ where } A=\begin{bmatrix}
 a & a' \\ 
 0 & a
\end{bmatrix}\in\widetilde{\mathcal{A}}.
$$
\end{rem}

\subsection{OVI Additive Convolution} 

For a $C^*$-OVI probability space $(\mathcal{A},\mathcal{B},E,E')$ and $x=x^*\in\mathcal{A}$, we consider its corresponding upper triangular probability space $(\widetilde{\mathcal{A}},\widetilde{\mathcal{B}},\widetilde{E})$. If we let $X=\begin{bmatrix}
x & 0 \\
0 & x
\end{bmatrix}\in\widetilde{\mathcal{A}}$ and $\widetilde{b}=\begin{bmatrix}
b & c \\
0 & b
\end{bmatrix}\in\widetilde{\mathcal{B}}$ where $b\in H^+(\mathcal{B})$ and $c\in\mathcal{B}$, then it is easy to see that 
\begin{eqnarray*}
G_{X}\left(\begin{bmatrix} b & c \\ 0 & b \end{bmatrix}\right)&=& 
\widetilde{E}\left( \Bigg\{
\begin{bmatrix}
 b & c \\ 0 & b
\end{bmatrix} -\begin{bmatrix}
x & 0 \\ 0 & x
\end{bmatrix} \Bigg\}^{-1}\right)   \\
&=& 
\widetilde{E}\left(\begin{bmatrix} (b-x)^{-1} & -(b-x)^{-1}c(b-x)^{-1} \\ 0 & (b-x)^{-1} \end{bmatrix}\right) \\
&=&\begin{bmatrix} E\left((b-x)^{-1}\right) & E'\left((b-x)^{-1}\right)-E\left((b-x)^{-1}c(b-x)^{-1}\right) \\ 0 & E\left((b-x)^{-1}\right) \end{bmatrix} \\
&=& \begin{bmatrix}
G_x(b) & G_x'(b)(c)+g_x(b) \\
0  & G_x(b)
\end{bmatrix}.
\end{eqnarray*}
Moreover, the $R$-transform of $X$ can be characterized as follows.
\begin{lem} 
There is $r>0$ so that for all $w\in B(0,r)$ and $v\in\mathcal{B}$, we have
$$
R_X\begin{bmatrix}
w & v \\
0 & w
\end{bmatrix}=
\begin{bmatrix}
R_x(w) & R_x'(w)(v)+r_x(w) \\
0 & R_x(w)
\end{bmatrix}
$$
where 
$$
r_x(w)= -(G_x^{\langle -1\rangle})'(w)\left(g_x(G_x^{\langle -1\rangle}(w))\right).
$$
\end{lem}
\begin{proof}

Note that there is $r>0$ and an open set $V_x\subset H^-(\mathcal{B})\cap B(0,r)$ with $0\in\overline{V_x}$ so that $G^{\langle -1 \rangle}_x$ is defined on $V_x$. 

Given $w\in V_x$ and $v\in\mathcal{B}$. Let us invert 
$\begin{bmatrix} w & v \\ 0 & w \end{bmatrix}=G_{X}\left(\begin{bmatrix} b & c \\ 0 & b \end{bmatrix}\right)$, which is equivalent to
$$
G_x(b)=w \text{ and } g_x(b)+G'_x(b)(c)=v.
$$
Thus, we have $b=G_x^{\langle -1\rangle}(w)$ and 
\begin{equation}\label{inveg}
g_x(b)+G'_x(b)(c)=v \Longrightarrow G_x'(b)(c)=v-g_x(b).
\end{equation}
Observe that 
\[
w=G_x(G_x^{\langle-1\rangle}(w)) \implies\text{Id}=G_x'(G_x^{\langle-1\rangle}(w))\circ
(G_x^{\langle-1\rangle})'(w). 
\]
Since $b=G_x^{\langle-1\rangle}(w)$, we get 
\[
G'_x(b)(c)=G'_x(G_x^{\langle-1\rangle}(w))(c)
=\left[(G_x^{\langle-1\rangle})'(w)\right]^{\langle-1\rangle}(c).
\]
Then \eqref{inveg} implies that 
\begin{eqnarray*}
&&\left[(G_x^{\langle-1\rangle})'(w)\right]^{\langle-1\rangle}(c) = v-g_x(G_x^{\langle-1\rangle}(w)) \\
&\Longrightarrow&  c = (G_x^{\langle-1\rangle})'(w)\left(v-g_x(G_x^{\langle-1\rangle}(w))\right).
\end{eqnarray*}
Hence, we obtain
$$
G_{X}^{\langle-1\rangle}\left(\begin{bmatrix} w & v \\ 0 & w 
\end{bmatrix}\right) =
\begin{bmatrix}
G_x^{\langle-1\rangle}(w) & (G_x^{\langle-1\rangle})'(w)\left(v-g_x(G_x^{\langle-1\rangle}(w))\right) \\
0 & G_x^{\langle-1\rangle}(w)
\end{bmatrix} 
$$
and also 
\begin{eqnarray*}
R_{X}\left(\begin{bmatrix} w & v \\ 0 & w \end{bmatrix}\right)&=&G_{X}^{\langle-1\rangle}\left(\begin{bmatrix} w & v \\ 0 & w 
\end{bmatrix}\right)-\begin{bmatrix} w & v \\ 0 & w \end{bmatrix}^{-1} \\
&=& G_{X}^{\langle-1\rangle}\left(\begin{bmatrix} w & v \\ 0 & w 
\end{bmatrix}\right)-
\begin{bmatrix} w^{-1} & -w^{-1}vw^{-1} \\ 
0 & w^{-1} 
\end{bmatrix} \\
&=& 
\begin{bmatrix}
G_x^{\langle-1\rangle}(w) -w^{-1} & (G_x^{\langle-1\rangle})'(w)\left(v-g_x(G_x^{\langle-1\rangle}(w))\right)+w^{-1}vw^{-1}\\
0 & G_x^{\langle-1\rangle}(w)-w^{-1}
\end{bmatrix}. \\
\end{eqnarray*}
By definition, $(G_x^{\langle-1\rangle})'(w)(\cdot)=-w^{-1}\cdot w^{-1}+R_x'(w)(\cdot)$ and therefore we conclude
\begin{eqnarray*}
&&R_{X}\left(\begin{bmatrix} w & v \\ 0 & w \end{bmatrix}\right) \\
&=& \begin{bmatrix} R_x(w) & R_x'(w)(v)-(G_x^{\langle -1\rangle})'(w)\left(g_x(G_x^{\langle -1\rangle}(w))\right) \\ 
0 & R_x(w) \end{bmatrix}
\end{eqnarray*}
where $w\in V_x$ and $v\in\mathcal{B}$. 

Since $R_X$ is well-defined and analytic on a neighborhood of the origin, the domain of $r_x$ above can be extended to a neighborhood of the origin and then we complete the proof. 
\end{proof}

\begin{rem}
We can see the formula of $r_x$ in previous Lemma provide us a connection between $r_x$ and $g_x$. We say the map $r_x$ the \emph{infinitesimal $R$-transform of $x$}. The scalar version of this result was found (see \cite{M}) and if we consider $\mathcal{B}=\mathbb{C}$ then it is easy to see that the formula coincides with the one in the scalar version. 
\end{rem}

Now, we state and prove our main theorem as follows.
\begin{thm} \label{thm4}
Given an $C^*$-OVI probability space $(\mathcal{A},\mathcal{B},E,E')$. If $x$ and $y$ are two self-adjoint elements in $\mathcal{A}$ that are infinitesimally free with respect to $(E,E')$, then for $b\in H^+(\mathcal{B})$ we have
\begin{eqnarray*}\lefteqn{
g_{x+y}(b) =  \left[G_{x+y}'(b)\circ\omega_{1}'(b)\circ  (G_{x+y}'(b))^{\langle-1\rangle}\right](g_{x}(\omega_1(b)))}
\\ & \qquad\qquad\qquad\qquad\mbox{} +
\left[G_{x+y}'(b)\circ\omega_{2}'(b)\circ (G_{x+y}'(b))^{\langle-1\rangle}\right](g_{y}(\omega_2(b)))
\end{eqnarray*}
where $\omega_1,\omega_2$ the subordination functions defined in \cite{v2000}. 
\end{thm}
\begin{proof}
Suppose that $x=x^*$ and $y=y^*$ are infinitesimally free with respect to $(E,E')$.
By Proposition \ref{thm1}, we have $X$ and $Y$ are free with respect to $\widetilde{E}$ where
$$
X=\begin{bmatrix}
x & 0 \\
0 & x
\end{bmatrix} \text{ and } Y=\begin{bmatrix}
y & 0 \\
0 & y
\end{bmatrix}. 
$$
Now, since $X$ and $Y$ are free with respect to $\widetilde{E}$, we have 
\begin{equation}\label{UR+formula}
R_X\Bigg(\begin{bmatrix}
w & v \\
0 & w
\end{bmatrix}\Bigg)+R_Y\Bigg(\begin{bmatrix}
w & v \\
0 & w
\end{bmatrix}\Bigg)=R_{X+Y}\Bigg(\begin{bmatrix}
w & v \\
0 & w
\end{bmatrix}\Bigg),
\end{equation}
for $v\in\mathcal{B}$ and $w\in B(0,r)$ for some $r>0$, which implies
$$
r_x(w)+r_y(w)=r_{x+y}(w).
$$
Thus, 
\begin{eqnarray*}
&&(G_x^{\langle -1\rangle})'(w)\left(g_x(G_x^{\langle -1\rangle}(w))\right) +(G_y^{\langle -1\rangle})'(w)\left(g_y(G_y^{\langle -1\rangle}(w))\right) \\
&=&(G_{x+y}^{\langle -1\rangle})'(w)\left(g_{x+y}(G_{x+y}^{\langle -1\rangle}(w))\right).
\end{eqnarray*}
Observe that 
$$
w=G_{x+y}(b)=G_x(\omega_1(b)) \Longrightarrow G_x^{\langle -1\rangle}(G_{x+y}(b))=\omega_1(b) 
$$
and hence 
\begin{eqnarray*}
&&(G_x^{\langle -1\rangle})'(G_{x+y}(b))\circ G_x'(\omega_1(b))\circ \omega_1'(b) (\cdot) = \omega_1'(b)(\cdot) \\
&\Rightarrow& (G_x^{\langle -1\rangle})'(G_{x+y}(b))\circ G_x'(\omega_1(b)) (\cdot) = Id \\
&\Rightarrow& (G_x'(\omega_1(b)))^{\langle -1\rangle} (\cdot) = (G_x^{\langle -1\rangle})'(G_{x+y}(b)) (\cdot).  
\end{eqnarray*}
Similarly, we have 
$$
(G_y'(\omega_2(b)))^{\langle -1\rangle} (\cdot) = (G_y^{\langle -1\rangle})'(G_{x+y}(b)) (\cdot).
$$
Therefore, we obtain
\begin{eqnarray*}
&&(G'_{x+y}(b))^{\langle -1\rangle}\left(g_{x+y}(b)\right) \\
&=&(G_x'(\omega_1(b)))^{\langle -1\rangle}\left(g_x(\omega_1(b))\right)+(G_y'(\omega_2(b)))^{\langle -1\rangle}\left(g_y(\omega_2(b))\right). 
\end{eqnarray*}
Now, we note that 
\begin{eqnarray*}
G_{x+y}(b)=G_x(\omega_1(b)) &\Longrightarrow & G_{x+y}'(b)(\cdot) = G_x'(\omega_1(b))\circ \omega'_1(b)(\cdot) \\
&\Longrightarrow & G_x'(\omega_1(b))\circ \omega'_1(b) \circ (G_{x+y}'(b))^{\langle -1\rangle} = Id \\
&\Longrightarrow& \omega'_1(b) \circ (G_{x+y}'(b))^{\langle -1\rangle} = (G_x'(\omega_1(b)))^{\langle -1\rangle}.
\end{eqnarray*}
Similarly, 
$$
\omega'_2(b) \circ (G_{x+y}'(b))^{\langle -1\rangle} = (G_y'(\omega_2(b)))^{\langle -1\rangle}.
$$
Therefore, we conclude that
\begin{eqnarray*}
g_{x+y}(b)& = & \left[G_{x+y}'(b)\circ\omega_{1}'(b)\circ  (G_{x+y}'(b))^{\langle-1\rangle}\right](g_{x}(\omega_1(b))) \\
& & \mbox{}+\left[G_{x+y}'(b)\circ\omega_{2}'(b)\circ (G_{x+y}'(b))^{\langle-1\rangle}\right](g_{y}(\omega_2(b))).
\end{eqnarray*}
\end{proof}

\subsection{OVI Multiplicative Convolution}

Suppose that $(\mathcal{A},\mathcal{B},E,E')$ is a $C^*$-OVI probability space. Let $x=x^*$ be an element in $\mathcal{A}$. If we let $X=\begin{bmatrix}
x & 0 \\ 0 & x 
\end{bmatrix}$, then for any $b,c\in\mathcal{B}$ with $$\left\|\begin{bmatrix}
b & c \\ 0 & b
\end{bmatrix}\right\|:=\|b\|+\|c\|<\|x\|^{-1},$$ we have
\begin{eqnarray*}
&& \Psi_{X}\left(\begin{bmatrix}
b & c \\
0 & b
\end{bmatrix}\right)\\ 
&=& 
\widetilde{E}\left(\left(\begin{bmatrix}
1 & 0 \\ 0 & 1
\end{bmatrix}-\begin{bmatrix}
b & c \\ 0 & b
\end{bmatrix}\begin{bmatrix}
x & 0 \\ 0 & x
\end{bmatrix}\right)^{-1}-\begin{bmatrix}
1 & 0 \\ 0 & 1
\end{bmatrix}\right) \\
&=& 
\widetilde{E}\left(\begin{bmatrix}
(1-bx)^{-1}-1 & (1-bx)^{-1}cx(1-bx)^{-1} \\ 0 & (1-bx)^{-1}-1
\end{bmatrix}\right) \\
&=& 
\begin{bmatrix}
E\left((1-bx)^{-1}-1\right) & E'((1-bx)^{-1}-1)+E((1-bx)^{-1}cx(1-bx)^{-1}) \\
0 & E((1-bx)^{-1}-1)
\end{bmatrix} \\
&=&
\begin{bmatrix}
\psi_x(b) & E'((1-bx)^{-1}-1)+\psi'_x(b)(c) \\
0 & \psi_x(b)
\end{bmatrix}.
\end{eqnarray*}
Note that 
$$
E'((1-bx)^{-1}-1)=\sum\limits_{n\geq 1} E'((bx)^n)
$$
is the infinitesimal moment-generating function of $x$, and we denote it by $\partial \psi_x$.  

Thus, we have 
$$
\Psi_{X}\left(\begin{bmatrix}
b & c \\
0 & b
\end{bmatrix}\right) = \begin{bmatrix}
\psi_x(b) & \psi_x'(b)(c)+\partial \psi_x(b) \\
0 & \psi_x(b)
\end{bmatrix}. 
$$
Suppose $(\mathcal{A},\mathcal{B},E,E')$ is a $C^*$- OVI probability space and $x=x^*\in \mathcal{A}$.  
If we assume that $E(x)$ is invertible, then the $S$-transform of $X$ can be described as the following Lemma.  
\begin{lem}\label{OVIM lemma}
For $w,v\in\mathcal{B}$ are small, we have 
$$
S_X\left(\begin{bmatrix}
w & v \\
0 & w
\end{bmatrix} \right) =\begin{bmatrix}
S_x(w) & S_x'(w)(v)+\partial S_x(w) \\
0 & S_x(w)
\end{bmatrix}
$$
where 
$$
\partial S_x(w)=-(\psi_{x}^{\langle -1\rangle})'(w)\left(\partial\psi_x(\psi_x^{\langle -1 \rangle}(w))\right).
$$
\end{lem}
\begin{proof}
Note that $\Psi_X^{\langle -1\rangle}$ is well-defined in an neighborhood of the origin, and let $w$ and $v$ in $\mathcal{B}$ be small so that 
$
\begin{bmatrix}
w & v \\ 0 & w
\end{bmatrix}
$ in such neighborhood. Then we consider 
$$
\Psi_X\left(\begin{bmatrix}
b & c \\
0 & b
\end{bmatrix}\right) = \begin{bmatrix}
w & v \\
0 & w
\end{bmatrix} \Longrightarrow 
\begin{bmatrix}
\psi_x(b) & \psi_x'(b)(c)+\partial \psi_x(b) \\
0 & \psi_x(b)
\end{bmatrix} = \begin{bmatrix}
w & v \\
0 & w
\end{bmatrix}. 
$$
Then we have 
$$
w=\psi_x(b) \Longrightarrow \psi^{\langle -1\rangle}_x(w)=b. 
$$
In addition, since 
$$
\psi_x(\psi_x^{\langle -1\rangle}(w))=w \Longrightarrow \psi_x'(\psi_x^{\langle -1\rangle}(w))\circ (\psi_x^{\langle -1\rangle})'(w) = Id, 
$$
we have 
$$
\psi_x'(b)(c) = [(\psi_{x}^{\langle -1\rangle})'(w)]^{\langle -1\rangle}(c). 
$$
Thus, 
\begin{eqnarray*}
v = \psi_x'(b)(c)+\partial \psi_x(b) &\Longrightarrow & \psi_x'(b)(c) =v-\partial \psi_x(b) \\
&\Longrightarrow & [(\psi_{x}^{\langle -1\rangle})'(w)]^{\langle -1\rangle}(c) = v-\partial \psi_x(b)  \\
&\Longrightarrow & c = (\psi_{x}^{\langle -1\rangle})'(w)(v-\partial\psi_x(b) ) \\
&\Longrightarrow & c = (\psi_{x}^{\langle -1\rangle})'(w)\left(v-\partial\psi_x(\psi_x^{\langle -1 \rangle}(w)) \right).
\end{eqnarray*}
Therefore, we have 
$$
\Psi_X^{\langle -1\rangle} \left(\begin{bmatrix}
w & v \\
0 & w
\end{bmatrix}\right) = \begin{bmatrix}
\psi_x^{\langle -1\rangle}(w) &  (\psi_{x}^{\langle -1\rangle})'(w)\left(v-\partial\psi_x(\psi_x^{\langle -1 \rangle}(w)) \right) \\
0 & \psi_x^{\langle -1\rangle}(w) 
\end{bmatrix},
$$
and then
$$
S_X\left(\begin{bmatrix}
w & v \\
0 & w
\end{bmatrix} \right) 
= \begin{bmatrix}
w^{-1}+1 & -w^{-1}vw^{-1} \\
0 & w^{-1}+1 \end{bmatrix} \Psi_X^{\langle -1\rangle} \left(\begin{bmatrix}
w & v \\
0 & w
\end{bmatrix}\right) 
= 
\begin{bmatrix}
S_x(w) & C_x(w,v)\\
0 & S_x(w)
\end{bmatrix} 
$$
where
\begin{eqnarray*}
&&C_x(w,v) \\
&=&(w^{-1}+1)(\psi_x^{\langle -1\rangle})'(w)(v)+(w^{-1}+1)'(v)\psi^{\langle -1\rangle}_x(w)-(\psi_{x}^{\langle -1\rangle})'(w)\left(\partial \psi_x(\psi_x^{\langle-1\rangle}(w))\right) \\
&=&
S_x'(w)(v)-(\psi_{x}^{\langle -1\rangle})'(w)\left(\partial \psi_x(\psi_x^{\langle-1\rangle}(w))\right).
\end{eqnarray*}
\end{proof}
\begin{defn}
We say the map $\partial S_x$ in Lemma \ref{OVIM lemma} the (operator-valued) \emph{infinitesimal $S$-transform} of $x$.  
\end{defn}
Now, let us state the main theorem. 
\begin{thm}\label{OVIMfcon}
Suppose that $(\mathcal{A},\mathcal{B},E,E')$ is a $C^*$-OVI probability space. Let $x=x^*$ and $y=y^*$ be two infinitesimally freely independent random variables in $\mathcal{A}$ such that $E(x)$ and $E(y)$ are invertible. Then for $\|w\|$ small enough, we have
\begin{eqnarray*}
&& \partial S_{xy}(w) \\ 
&=& S_y(w) S_x'(S_y(w)^{-1}wS_y(w)) \left( S_y(w)^{-1}w\partial S_y(w)-S_y(w)^{-1}\partial S_y(w)S_y(w)^{-1}wS_y(w) \right) \\
 & & + S_y(w) \partial S_x(S_y(w)^{-1}wS_y(w)) + \partial S_y(w) S_x(S_y(w)^{-1}wS_y(w)). 
\end{eqnarray*}
\end{thm}
\begin{proof}
Note that by Proposition \ref{thm1}, 
$X=
\begin{bmatrix} 
x & 0 \\
0 & x 
\end{bmatrix}$ and 
$
Y=\begin{bmatrix}
y & 0 \\
0 & y
\end{bmatrix}
$ are free with respect to $\widetilde{E}$. By \eqref{Sconvolution}, we have 
\begin{equation}\label{upp S convolution}
S_{XY}\left(\begin{bmatrix}w & v \\ 0 & w \end{bmatrix}\right) = S_{Y}\left(\begin{bmatrix}w & v \\ 0 & w \end{bmatrix}\right) S_{X}\left(S_{Y}\left(\begin{bmatrix}w & v \\ 0 & w \end{bmatrix}\right)^{-1}\begin{bmatrix}w & v \\ 0 & w \end{bmatrix}S_{Y}\left(\begin{bmatrix}w & v \\ 0 & w \end{bmatrix}\right) \right) 
\end{equation}
for $\|w\|$ and $\|v\|$ small enough. 
The left hand side of \eqref{upp S convolution} is 
$$
\begin{bmatrix}
S_{xy}(w) & S_{xy}'(w)(v)+\partial S_{xy}(w) \\ 0 & S_{xy}(w)
\end{bmatrix}.
$$
To compute the right hand side of \eqref{upp S convolution}, we first compute
\begin{eqnarray*}
& & S_{Y}\left(\begin{bmatrix}w & v \\ 0 & w \end{bmatrix}\right)^{-1}\begin{bmatrix}w & v \\ 0 & w \end{bmatrix}S_{Y}\left(\begin{bmatrix}w & v \\ 0 & w \end{bmatrix}\right) \\
&=& 
\begin{bmatrix}
S_y(w)^{-1} & -S_y(w)^{-1}(S_y'(w)(v)+\partial S_y(w))S_y(w)^{-1} \\
0 & S_y(w)^{-1}
\end{bmatrix} \\
&&
\begin{bmatrix}
w & v \\
0 & w 
\end{bmatrix}
\begin{bmatrix}
S_y(w) & S_y'(w)(v)+\partial S_y(w) \\
0 & S_y(w)
\end{bmatrix} \\
&=&
\begin{bmatrix}
S_y(w)^{-1} &  ([S_y(w)]^{-1})'(v) -S_y(w)^{-1}\partial S_y(w)S_y(w)^{-1} \\
0 & S_y(w)^{-1} 
\end{bmatrix} \\
&&
\begin{bmatrix}
wS_y(w) & wS_y'(w)(v)+w\partial S_y(w)+vS_y(w) \\
0 & wS_y(w)
\end{bmatrix}, 
\end{eqnarray*}
thus the $(1,1)$-entry of 
$$
S_{Y}\left(\begin{bmatrix}w & v \\ 0 & w \end{bmatrix}\right)^{-1}\begin{bmatrix}w & v \\ 0 & w \end{bmatrix}S_{Y}\left(\begin{bmatrix}w & v \\ 0 & w \end{bmatrix}\right) 
$$
is $S_y(w)^{-1}wS_y(w)$ and the $(1,2)$-entry is the follows 
\begin{eqnarray*}
&&S_y(w)^{-1} \left( wS_y'(w)(v)+w\partial S_y(w)+vS_y(w) \right) + \\
&& \left(  ([S_y(w)]^{-1})'(v) -S_y(w)^{-1}\partial S_y(w)S_y(w)^{-1} \right) wS_y(w) \\
&=& S_y(w)^{-1}(wS_y'(w)(v) + vS_y(w)) + ([S_y(w)]^{-1})'(v) wS_y(w) \\
& & + S_y(w)^{-1}w\partial S_y(w)-S_y(w)^{-1}\partial S_y(w)S_y(w)^{-1}wS_y(w). 
\end{eqnarray*}
Therefore, the $(1,1)$-entry of
$$
S_{X}\left(S_{Y}\left(\begin{bmatrix}w & v \\ 0 & w \end{bmatrix}\right)^{-1}\begin{bmatrix}w & v \\ 0 & w \end{bmatrix}S_{Y}\left(\begin{bmatrix}w & v \\ 0 & w \end{bmatrix}\right) \right)
$$
is 
$S_x(S_y(w)^{-1}wS_y(w))$ and the $(1,2)$-entry is 
\begin{eqnarray*}
& & S_x'(S_y(w)^{-1}wS_y(w))\left( S_y(w)^{-1}(wS_y'(w)(v) + vS_y(w)) + ([S_y(w)]^{-1})'(v) wS_y(w) \right)  \\
&+& S_x'(S_y(w)^{-1}wS_y(w)) \left( S_y(w)^{-1}w\partial S_y(w)-S_y(w)^{-1}\partial S_y(w)S_y(w)^{-1}wS_y(w) \right) \\
&+& \partial S_x(S_y(w)^{-1}wS_y(w)).
\end{eqnarray*}
Hence, comparing the $(1,1)$-entry on both hand side of \eqref{upp S convolution}, we have
$$
S_{xy}(w) = S_y(w)S_x(S_y(w)^{-1}wS_y(w)).
$$
In addition, let us compare the $(1,2)$-entry on both side of \eqref{upp S convolution} with $v=0$, then  
\begin{eqnarray*}
&&\partial S_{xy}(w) \\
&=&S_y(w) S_x'(S_y(w)^{-1}wS_y(w)) \left( S_y(w)^{-1}w\partial S_y(w)-S_y(w)^{-1}\partial S_y(w)S_y(w)^{-1}wS_y(w) \right) \\
 &&+ S_y(w) \partial S_x(S_y(w)^{-1}wS_y(w)) + \partial S_y(w) S_x(S_y(w)^{-1}wS_y(w)). 
\end{eqnarray*}
\end{proof}
\begin{rem}
When $\mathcal{B}=\mathbb{C}$, we obtain the formula of scalar version of infinitesimal $S$-transform and the infinitesimal free multiplicative convolution as follows. 

For a given random variable $x\in\mathcal{A}$, the infinitesimal $S$-transform of $x$ is defined on an neighborhood of the origin by
$$ 
\partial S_x(w)=-(\psi_x^{\langle-1\rangle})'(w)\partial \psi_x\left(\psi^{\langle -1 \rangle}_x(w)\right) 
$$
where 
$$
\partial \psi_x(z):=\varphi'_x\left((1-zx)^{-1}-1\right)=\sum\limits_{n\geq 1}\varphi'(x^n)z^n \qquad \text{ for }|z|<\frac{1}{\|x\|}.
$$
Moreover, if $x$ and $y$ are infinitesimally freely independent, then 
\begin{equation}\label{smultip}
\partial S_{xy}(w)=\partial S_x(w)S_y(w)+S_x(w)\partial S_y(w)  \text{ for }\|w\| \text{ small}.    
\end{equation}
\end{rem}
Note that the $S$-transform of type $B$ is introduced in Popa, and \eqref{smultip} can also be derived from Theorem \cite[Theorem 4.2]{POP10}.  



\bibliographystyle{abbrv}
\bibliography{main.bib}
\end{document}